\documentclass[amssymb,twoside,11pt]{article}
\usepackage{latexsym,amsmath,graphicx}
\usepackage[dvipsnames]{xcolor}
\usepackage{amsfonts}
\newtheorem{theorem}{Theorem}[section]
\newtheorem{Lemma}[theorem]{{\bf Lemma}}

\newtheorem{cor}[theorem]{{\bf Corollary}}
\newtheorem{rem}[theorem]{{\bf Remark}}

\newtheorem{definition}{Definition}[section]
\numberwithin{equation}{section}
\newenvironment{proof}{\indent{\em Proof:}}{\quad \hfill
$\Box$\vspace*{2ex}}

\begin{document}
\setcounter{page}{1}
\begin{center}
\vspace{0.4cm} {\large{\bf Analysis of Nonlinear  Fractional Differential Equations Involving Atangana--Baleanu--Caputo   Derivative  }} \\
\vspace{0.5cm}
Kishor D. Kucche $^{1}$ \\
kdkucche@gmail.com \\
\vspace{0.35cm}
Sagar T. Sutar  $^{2}$\\
sutar.sagar007@gmail.com\\

\vspace{0.35cm}
$^{1}$ Department of Mathematics, Shivaji University, Kolhapur-416 004, Maharashtra, India.\\
$^2$ Department of Mathematics, Vivekanand College (Autonomous), Kolhapur-416003, Maharashtra, India.\\

\end{center}

\def\baselinestretch{1.0}\small\normalsize

\begin{abstract}
In the present paper, we determine the estimations on Atangana--Baleanu--Caputo fractional derivative at extreme points.  With the assistance of the estimations obtained,  we derive the comparison results. Peano's type existence results established for nonlinear fractional differential equations involving Atangana--Baleanu--Caputo fractional derivative. The acquired comparison results are then utilized to deal  with the existence of local, extremal and global solution.
\end{abstract}
\def\baselinestretch{1.5}
\allowdisplaybreaks

\section{Introduction}
The fundamental theory of fractional calculus has been exhibited principally with the two fractional derivative operators Riemann-Liouville and Caputo fractional derivatives. The Riemann-Liouville (RL) fractional derivative plays an important role mainly in the development of the theory of fractional derivatives, integrals and for application to develop various theories in pure mathematics. The RL fractional integral of the Caputo fractional derivative generates the initial conditions in the form of classical derivatives. As a result of which we have a clear physical interpretation in the modeling of real-world phenomena in the form of fractional differential equations (FDEs). Basic calculus and the theoretical development FDEs involving these two fractional derivative operators have been excellently presented in the monographs \cite{Miller,Pod,Kilbas,Laksh6,Diethelm}. 

Kai et al. \cite{KD3} researched the existence and uniqueness of solutions, structural stability and the dependence of the solution on the order of the differential equation and on the initial condition. Daftardar-Gejji et al. \cite{VSG} broadened these investigations for the system of fractional differential equations. Lakshmikantham and  Vatsala \cite{Laksh2,Laksh5} built up the theory of fractional differential and integral inequalities and employed it to explore the existence of extremal  and global solutions of nonlinear FDEs. Basic development relating to investigations on the theory of nonlinear FDEs can be found in \cite{Laksh1,Laksh3,RPA,JRW1} and ongoing advancements can be found in \cite{NDC,GIS} and references therein. 

Researchers working in the field of applied mathematics have been developing the theory of fractional calculus in various directions by defining different types of arbitrary order derivatives and integrals. It is very well known that the traditional fractional derivative operators provide better mathematical modeling of many real-world phenomena than the classical integer derivatives. In spite of this reality, numerous researchers accept that that worthy exactness may not be accomplished in the modeling of physical phenomena involving memory effect in the whole of the time duration due to the presence of a singular kernel in the definition of traditional fractional derivatives. 

To eliminate the singular kernel, a non-singular fractional derivative operator with exponential  kernel is proposed in \cite{Cap1} ,  which is well known as Caputo--Fabrizio (CF) fractional derivative. For the development of the theory relating to FDEs involving CF derivative and it's real-world application one can refer\cite{Dbaleu4,MSAydo,Losda,Atanga} and the references given therein. Motivated by the investigations of \cite{Cap1}, Atangana and Baleanu in \cite{Atan3} proposed new fractional derivative  having  Mittag-Leffler (ML) function as its kernel, which is well known as Atangana--Baleanu--Caputo (ABC) fractional derivative. The basic calculus of ABC-fractional derivative can be found in \cite{Bal1,Dbale1,Fer,Moh,Atan5}. Non-locality of ABC-fractional derivative with singular ML kernel effectively permits taking care of the nonlocal dynamics, computational purposes and capturing the various features of realistic systems more suitably. Mathematical modeling via ABC-FDEs of the various outbreak, such as dengue fever, the free motion of a coupled oscillator, a tumor-immune surveillance mechanism etc. and efficient numerical method to tackle this has been explored in \cite{Amin,Amin1, Dbale2,Amin2,Sume,MSA}.

Jarad et al. \cite{Jarad} provided sufficient conditions for the existence and uniqueness of the solution of nonlinear ABC-FDEs. Authors derived  the Gronwall inequality in the frame of Atangana-Baleanu fractional integral and through it investigated Ulam-Hyers stability of nonlinear ABC-FDEs. Baleanu et al.\cite{Bale4} considered the existence and uniqueness of solution nonlinear ABC-FDEs and structured a numerical procedure dependent on the fractional Euler and predictor-corrector technique.
Syam et al.\cite{Syam}  determined existence and uniqueness results for the linear and nonlinear ABC-FDEs and exhibited a numerical technique dependent on the Cheby-shev collocation technique. Afshari et al. \cite{Af}  demonstrated the existence results for ABC-FDEs utilizing the fixed point theorems for contractive mappings such as $\alpha$-$\gamma$--Geraghty type, $\alpha$-type $\mathcal{F}$-contraction in $\mathcal{F}$-complete metric space.
Ravichandran et al.  \cite{Ravi1,Ravi2,Ravi3,Ravi4} examined the existence and uniqueness of solution for ABC-fractional differential and integrodifferential equations. Shah et al. \cite{shah} analyzed the qualitative theory of existence and stability theory of Ulam’s type for evolution ABC-FDEs.

Even if, many researchers have investigated nonlinear ABC-FDEs, it ought to be seriously analyzed for the qualitative properties. In this view, motivated by the applications and the interesting literature on ABC-FDEs, on the line of  \cite{Laksh2,Laksh5,Moh}, we investigates estimation on ABC-fractional derivatives at extreme points,  comparison results,  local and global existence of a solution and  extremal solution for nonlinear ABC-FDEs of the form 
\begin{align}
^{ABC}{_0\mathcal{D}}_\tau^\alpha \omega(\tau)&= f\left( \tau, \omega(\tau)\right), \tau\in J=[0, T],\; T>0, \label{ABCe1.1} \\
\omega(0)&=\omega_0,\label{ABCe1.2}
\end{align}
where $0<\alpha<1, ~^{ABC}{_0\mathcal{D}}_\tau^\alpha$ is the ABC-  fractional derivative operator, $\omega,~^{ABC}{_0\mathcal{D}}_\tau^\alpha \omega\in C(J)$  and  $f\in C(J\times \mathbb{R},\mathbb{R})$ is continuous non-linear function. The local existence  of  solution to ABC-FDEs \eqref{ABCe1.1}--\eqref{ABCe1.2} is  based on Peano's theorem and the comparison results which we have derived in the present paper.

The novelty of the present paper is that we have obtained comparison results, local and global existence of a solution, and extremal solution without demanding the monotonicity and Holder continuity assumption on the nonlinear function associated with  ABC-FDEs  \eqref{ABCe1.1}--\eqref{ABCe1.2}. Further, we provided alternative proofs to few results of \cite{Moh} pertain to fractional integral inequalities.

This paper is organized as follows. In  section 2, we recall basic definitions and results related with ABC-fractional derivative. Section 3 deals with estimation on ABC-fractional derivative at extreme points. In  section 4, we derive comparison results for ABC-FDEs involving initial and boundary conditions.  Section 5 deals with local existence and extremal of solution of ABC-FDEs  \eqref{ABCe1.1}--\eqref{ABCe1.2}.  In section 5,  the global existence of solution is proved for ABC-FDEs \eqref{ABCe1.1}--\eqref{ABCe1.2}.

\section{Preliminaries}
In this section, we recall some definitions and basic results about ABC-fractional derivative operator and  generalized Mittag--Leffler function.
\begin{definition} {\normalfont\cite{Syam}}
Let $p\in[1,\infty)$ and $\Omega$ be an open subset of $\mathbb{R}$ the Sobolev space $H^p(\Omega)$ is defined as 
$$ H^p(\Omega)=\left\lbrace f\in L^2(\Omega):D^\beta f\in L^2(\Omega), \text{for all } |\beta|\leq p\right\rbrace .$$
\end{definition}
\begin{definition}{\normalfont\cite{Atan3}}
Let $\omega\in H^1(0,T)$ and $\alpha \in [0,1]$, the left Atangana--Baleanu--Caputo fractional derivative of $\omega$ of order $\alpha $ is defined  by
\begin{equation*}
^{ABC}{_0}\mathcal{D}_\tau^\alpha \omega(\tau)=\dfrac{B(\alpha)}{1-\alpha}\int_{0}^\tau\mathbb{E}_\alpha\left[ -\dfrac{\alpha}{1-\alpha}(\tau-\sigma)^\alpha\right] \omega'(\sigma)d\sigma,
\end{equation*} 
where $B(\alpha)>0$ is a normalization function satisfying $B(0)=B(1)=1$ and $\mathbb{E}_\alpha$ is one parameter Mittag-Leffler function  {\normalfont\cite{Kilbas,Diethelm}}  defined by 
$$
\mathbb{E}_\alpha (z)=\sum_{n=0}^{n=\infty}\frac{z^n}{\Gamma(n \alpha+1)}.
$$
The associated fractional integral is defined  by
\begin{equation*}
^{AB}{_0}I_\tau^\alpha \omega(\tau)=\dfrac{1-\alpha}{B(\alpha)} \omega(\tau)+\dfrac{\alpha}{B(\alpha)}~{_0}I_\tau^\alpha \omega(\tau).
\end{equation*} 
where 
$${_0}I_\tau^\alpha \omega(\tau)= \dfrac{1}{\Gamma(\alpha)}\int_{0}^\tau(\tau-\sigma)^{\alpha-1}\omega(\sigma)d\sigma,
$$
is the Riemann--Liouville fractional integral {\normalfont\cite{Kilbas,Diethelm}} of $\omega$ of order $\alpha$.
\end{definition} 
\begin{Lemma}{\normalfont\cite{Bal1}} If $0<\alpha<1$,~ then 
$
^{AB}{_0}I_\tau^\alpha\; \left( ^{ABC}{_0}\mathcal{D}_\tau^\alpha \omega(\tau)\right) =\omega(\tau)-\omega(0).
$
\end{Lemma}

\begin{definition}{\normalfont\cite{pra,kilb,Erd}}
The generalized Mittag-Leffler function $ \mathbb{E}_{\alpha,\beta}^\gamma(z) $ for the complex numbers $\alpha,\beta,\gamma $ with Re($\alpha)>0$ is defined as 
$$
 \mathbb{E}_{\alpha,\beta}^\gamma(z)=\sum_{k=0}^{\infty}\dfrac{(\gamma)_k}{\Gamma(\alpha k+\beta)}\dfrac{z^k}{k!},
$$
where $(\gamma)_k$ is the Pochhammer symbol given by
$$
(\gamma)_0=1,\;(\gamma)_k=\gamma(\gamma+1)\cdots(\gamma+k-1),\; k=1,2,\cdots
$$
\end{definition}
Note that,
$$
\mathbb{E}_{\alpha,\beta}^1(z) =\mathbb{E}_{\alpha,\beta}(z) ~\mbox{and }~\mathbb{E}_{\alpha,1}^1(z) =\mathbb{E}_{\alpha}(z).
$$
\begin{Lemma}{\normalfont\cite{Bal1}}\label{ABCLm2.2} 
Let  $0<\alpha<1$ and $\beta,\sigma,\lambda\in\mathbb{C} \left( Re(\beta)>0 \right)$. Then
$$
^{ABC}{_0\mathcal{D}}_\tau^\alpha \left[\tau^{\beta-1}\;\mathbb{E}_{\alpha,\, \beta}^\sigma \, (\lambda \, \tau^\alpha) \right]=\dfrac{B(\alpha)}{1-\alpha} \; \tau^{\beta-1}\;\mathbb{E}_{\alpha,\,\beta}^{1+\sigma}(\lambda \,\tau^\alpha).
$$
\end{Lemma}
\section{Estimates on ABC fractional derivatives at extreme points}

\begin{theorem}\label{ABCThm3.1}
If $m$ is any differentiable function defined on $J$ such that $^{ABC}{_0\mathcal{D}}_\tau^\alpha m\in C(J)$  and there exists $\tau_0\in (0,T]$ with $m(\tau_0)= 0,\;m(\tau)\leq 0,\; \tau\in [0,\tau_0)$,  then  
$^{ABC}{_0\mathcal{D}}_\tau^\alpha m(\tau_0)\geq 0$.
\end{theorem}
\begin{proof}
Using integration by parts we write
\begin{align*}
&^{ABC}{_0\mathcal{D}}_\tau^\alpha m(\tau)=\dfrac{B(\alpha)}{1-\alpha}\int_{0}^\tau\mathbb{E}_\alpha\left( -\dfrac{\alpha}{1-\alpha}(\tau-\sigma)^\alpha\right) m'(\sigma)d\sigma\\
&=\dfrac{B(\alpha)}{1-\alpha}\left\{ \left[ \mathbb{E}_\alpha\left( -\dfrac{\alpha}{1-\alpha}(\tau-\sigma)^\alpha\right) m(\sigma)  \right]  _{\sigma=0}^{\sigma=\tau} -\int_{0}^\tau\left( \dfrac{d}{d\sigma}\mathbb{E}_\alpha\left( -\dfrac{\alpha}{1-\alpha}(\tau-\sigma)^\alpha\right) \right) m(\sigma)d\sigma\right\}\\
&=\dfrac{B(\alpha)}{1-\alpha}\left\lbrace  m(\tau)-\mathbb{E}_\alpha\left ( -\dfrac{\alpha}{1-\alpha}\tau^\alpha\right) m(0)  -\int_{0}^\tau\left( \dfrac{d}{d\sigma}\mathbb{E}_\alpha\left( -\dfrac{\alpha}{1-\alpha}(\tau-\sigma)^\alpha\right) \right) m(\sigma)d\sigma\right\rbrace .
\end{align*} 
Since $m(\tau_0)=0$, we have
\begin{align}\label{ABCe3.1}
^{ABC}{_0\mathcal{D}}_\tau^\alpha m(\tau_0)&=-\dfrac{B(\alpha)}{1-\alpha}\left\lbrace \mathbb{E}_\alpha\left ( -\dfrac{\alpha}{1-\alpha}\tau_0^\alpha\right) m(0) +\int_{0}^{\tau_0}\left( \dfrac{d}{d\sigma}\mathbb{E}_\alpha\left[ -\dfrac{\alpha}{1-\alpha}(\tau_0-\sigma)^\alpha\right] \right) m(\sigma)d\sigma\right\rbrace .
\end{align} 
From \cite{Moh}, we have
$$
\mathbb{E}_\alpha(-\tau^\alpha)= \int_{0}^{\infty}e^{-r\tau}K_\alpha(r)dr, \;0<\alpha<1,\;\mbox{for all} ~\tau>0,
$$
where 
$$
K_\alpha(r)=\dfrac{1}{\pi}\dfrac{r^{\alpha-1}\sin(\alpha\pi)}{r^{2\alpha}+2r^\alpha\cos(\alpha\pi)+1}>0.
$$
Since
$
K_\alpha(r), ~e^{- r \,\left ( \frac{\alpha}{1-\alpha}\right)^{\frac{1}{\alpha}} \tau_0 }>0,\;\text{for}\; r>0 ~\mbox{and}~ 0<\alpha<1,
$, we have
\begin{equation} \label{kk1}
\mathbb{E}_\alpha\left ( -\dfrac{\alpha}{1-\alpha}\tau_0^\alpha\right)=\mathbb{E}_\alpha\left ( -\left( \left[ \dfrac{\alpha}{1-\alpha}\right]^\frac{1}{\alpha} \tau_0\right)^\alpha\; \right)=\int_{0}^{\infty}e^{- r \,\left ( \frac{\alpha}{1-\alpha}\right)^{\frac{1}{\alpha}} \tau_0 } \, K_\alpha(r)\,dr> 0.
\end{equation}
Since, $B(\alpha)>0$ and $m(0)\leq 0$, from \eqref{kk1} it follows that 
\begin{equation}\label{ABCe3.2}
-\dfrac{B(\alpha)}{1-\alpha}\mathbb{E}_\alpha\left( -\dfrac{\alpha}{1-\alpha}\tau_0^\alpha\right) m(0)\geq 0.
\end{equation}
In view of inequality \eqref{ABCe3.2}, Eq. \eqref{ABCe3.1} reduces 
to
$$ 
^{ABC}{_0\mathcal{D}}_\tau^\alpha m(\tau_0)\geq  -\dfrac{B(\alpha)}{1-\alpha}\int_{0}^{\tau_0}\left( \dfrac{d}{d\sigma}\mathbb{E}_\alpha\left ( -\dfrac{\alpha}{1-\alpha}(\tau_0-\sigma)^\alpha\right) \right) m(\sigma)d\sigma.
$$
Again from \cite{Moh}, we have
 $$
\dfrac{d}{d\tau}\mathbb{E}_\alpha\left(- \dfrac{\alpha}{1-\alpha}(\tau_0-\tau)^\alpha\right)  \geq 0, 
$$
Therefore, $m(\tau)\leq 0, \;\tau\in [0,\tau_0)
$ gives 
\begin{equation} \label{kk2}
\dfrac{d}{d\tau}\mathbb{E}_\alpha\left[ -\dfrac{\alpha}{1-\alpha}(\tau_0-\tau)^\alpha\right](-m(\tau))\geq 0, ~\tau\in [0,\tau_0).
\end{equation}
Using the inequality \eqref{kk2} and the fact $ B(\alpha)>0$, it follows that
$$
^{ABC}{_0\mathcal{D}}_\tau^\alpha m(\tau_0)\geq \dfrac{B(\alpha)}{1-\alpha}\int_{0}^{\tau_0}\left( \dfrac{d}{d\sigma}\mathbb{E}_\alpha\left(- \dfrac{\alpha}{1-\alpha}(\tau_0-\sigma)^\alpha\right) \right) \left( -m(\sigma)\right) d\sigma\geq 0.
$$ 
This completes the proof.
\end{proof}

The dual of the Theorem \ref{ABCThm3.1} is  also hold.
\begin{theorem}\label{ABCThm3.2}
If $m$ is any differentiable function defined on $J$ such that $^{ABC}{_0\mathcal{D}}_\tau^\alpha m\in C(J)$ and there exists $\tau_0\in (0,T]$ with $m(\tau_0)= 0, \;m(\tau)\geq 0,\; \tau\in [0,\tau_0)$,  then  
$^{ABC}{_0\mathcal{D}}_\tau^\alpha m(\tau_0)\leq 0$.
\end{theorem} 
\begin{proof}
One can observe that if $m(\tau)$ satisfies the assumptions of Theorem \ref{ABCThm3.2}, then $(-m)(\tau)$ satisfies the conditions of Theorem \ref{ABCThm3.1}. Hence by applying Theorem \ref{ABCThm3.1} with $m$ replaced by $(-m)$ we obtain, $^{ABC}{_0\mathcal{D}}_\tau^\alpha (-m)(\tau_0)\geq 0$.
This gives
$$
^{ABC}{_0\mathcal{D}}_\tau^\alpha m(\tau_0)\leq 0.
$$
\end{proof}

In the following Theorem, we give an alternative  proof of Lemma 2.1  \cite{Moh} utilizing the outcome that we have gotten  in Theorem \ref{ABCThm3.1}.
\begin{theorem}\label{ABCThm3.3}
Let a differentiable  function $f$ defined on $J$   attain its maximum at a point $\tau_0\in J$ and \; $^{ABC}{_0\mathcal{D}}_\tau^\alpha f\in C(J)$. Then the inequality
\;$^{ABC}{_0\mathcal{D}}_\tau^\alpha f(\tau_0)\geq 0$
holds true.
\end{theorem}
\begin{proof} Let $f_{\max}=f(\tau_0)=\underset{\tau\in J}{\max}\;f(\tau)$. Define $ m(\tau)=f(\tau)-f_{\max},~\tau\in J$. Then $m$ is differentiable function defined on $J$ such that, $^{ABC}{_0\mathcal{D}}_\tau^\alpha m\in C(J)$ and 
$$
m(\tau_0)=0,\;m(\tau)=f(\tau)-f_{\max}<0,\; \text{for all}\; \tau\in J \, \backslash \,\{\tau_0\} 
$$
Therefore, $m(\tau_0)=0$ and $ m(\tau)<0,\; \text{for all}\; \tau\in [0,\tau_0).$
Since $m$ satisfies all the conditions of  Theorem \ref{ABCThm3.1}, by applying it we obtain $^{ABC}{_0\mathcal{D}}_\tau^\alpha m(\tau_0)\geq 0$.
Since,
$$
^{ABC}{_0\mathcal{D}}_\tau^\alpha m(\tau)=\;^{ABC}{_0\mathcal{D}}_\tau^\alpha \left(f(\tau)-f_{\max}\right)=\;^{ABC}{_0\mathcal{D}}_\tau^\alpha f(\tau),\;\tau\in J,
$$
we have
$$
^{ABC}{_0\mathcal{D}}_\tau^\alpha f(\tau_0)=\;^{ABC}{_0\mathcal{D}}_\tau^\alpha m(\tau_0)\geq 0.
$$
\end{proof}

In the next  Theorem, we provide an alternating proof of Lemma 2.2 \cite{Moh}.
\begin{theorem}
Let a differentiable  function $f$ defined on $J$   attain its minimum at a point $\tau_0\in J$ and\;  $^{ABC}{_0\mathcal{D}}_\tau^\alpha f\in C(J)$. Then the inequality
\;$^{ABC}{_0\mathcal{D}}_\tau^\alpha f(\tau_0)\leq 0$
holds true.
\end{theorem}
\begin{proof}
Let $f_{\min}=f(\tau_0)=\underset{\tau\in J}{\min}f(\tau)$ and define $ m(\tau)=f(\tau)-f_{\min},~\tau\in J$. Then, one can complete the remaining proof by applying Theorem \ref{ABCThm3.2} and following the similar types of steps  as in the proof of Theorem \ref{ABCThm3.3}.
\end{proof}

\section{Comparison Results} 
\begin{theorem}\label{ABCThm3.5}
Let $f\in C(J\times\mathbb{R},\mathbb{R})$. Let  $v,w$ be any differentiable functions on $J$ such that $^{ABC}{_0\mathcal{D}}_\tau^\alpha\; v,~^{ABC}{_0\mathcal{D}}_\tau^\alpha w\in C(J)$,  satisfying
\item[\normalfont(i)] $^{ABC}{_0\mathcal{D}}_\tau^\alpha v(\tau)\leq  f\left( \tau, v(\tau)\right),\; \tau\in J, $
\item[\normalfont(ii)] $^{ABC}{_0\mathcal{D}}_\tau^\alpha w(\tau)\geq  f\left( \tau, w(\tau)\right),\; \tau\in J,$

\noindent one of the above inequalities being strict. 

\noindent Then $v(0)<w(0)$, implies
$$
v(\tau)< w(\tau),\;\tau\in J.
$$
\end{theorem}
\begin{proof}
Suppose that the conclusion of the theorem does not holds. Then by continuity of $v,w$ there exits  $\tau_0\in J$ such that
$$
v(\tau_0)=w(\tau_0)\; \text{and}\; v(\tau)< w(\tau)\;\text{for all} \; \tau\in[0,\tau_0).
$$

\noindent Define $m(\tau)=v(\tau)-w(\tau),\; \tau\in J$.  Then the function $m(\tau)$ is differentiable on $J$ with $^{ABC}{_0\mathcal{D}}_\tau^\alpha m\in C(J)$ and $\tau_0\in J$ is such that
$$
m(\tau_0)=0\; \text{and}\; m(\tau)< 0\;\text{for all} \; \tau\in[0,\tau_0).
$$
Since $m$ satisfies all assumptions of Theorem \ref{ABCThm3.1}, we get  $^{ABC}{_0\mathcal{D}}_\tau^\alpha m(\tau_0)\geq 0$. 

\noindent This gives 
$$
^{ABC}{_0\mathcal{D}}_\tau^\alpha v(\tau_0)\geq\; ^{ABC}{_0\mathcal{D}}_\tau^\alpha w(\tau_0).
$$

\noindent Suppose that the inequality (i)  is strict, then we get
$$
f\left( \tau_0, v(\tau_0)\right)>\; ^{ABC}{_0\mathcal{D}}_\tau^\alpha v(\tau_0)\geq\; ^{ABC}{_0\mathcal{D}}_\tau^\alpha w(\tau_0) \geq f\left( \tau_0, w(\tau_0)\right).
$$

\noindent This is  contradiction with  $v(\tau_0)=w(\tau_0)$. Therefore, we must have 
$$
v(\tau)<w(\tau),\;\text{for all}\; \tau\in J.
$$
This completes the proof of theorem.
\end{proof}
% % % % % % %
\begin{theorem}\label{ABCThm3.6}
Assume  that the conditions of Theorem \ref{ABCThm3.5} holds {\normalfont(ii)}. Suppose that
$$
 f(\tau,\omega)-f(\tau,\eta)\leq L(\omega-\eta),\; \text{for all }\; \omega,\eta\in \mathbb{R}\;\text{with}\;\omega\geq \eta\;\text{and}\; 0<L<\dfrac{B(\alpha)}{1-\alpha}.
$$
Then $v(0)\leq w(0)$ implies
$$
v(\tau)\leq w(\tau),\; \text{for all}\;\tau\in J.
$$
\end{theorem}
\begin{proof}
For $\epsilon>0$, we define
\begin{equation}\label{ABCe3.3}
w_\epsilon(\tau)=w(\tau)+\epsilon\mathbb{E}_{\alpha}(\tau^\alpha),\; \tau\in J.
\end{equation}
By choice of $w$ and  Lemma \ref {ABCLm2.2},  the function $ w_\epsilon$ is differentiable on $J$ such that $^{ABC}{_0\mathcal{D}}_\tau^\alpha w_\epsilon\in C(J)$ 
and 
$$
w_\epsilon(0)=w_0+\epsilon> w(0).
$$
Since $w_\epsilon(\tau)\geq w(\tau),\; \tau\in J$, by using  Lipschitz condition on $f$, we have
$$
f(\tau,w_\epsilon(\tau)) -f(\tau,w(\tau))\leq L(w_\epsilon(\tau)-w(\tau))=L\epsilon\mathbb{E}_{\alpha}(\tau^\alpha).
$$
Using the condition on $L$, we have 
\begin{equation}\label{ABCe3.4}
f(\tau,w(\tau))\geq f(\tau,w_\epsilon(\tau))-L\epsilon\mathbb{E}_{\alpha}(\tau^\alpha)>f(\tau,w_\epsilon(\tau))-\dfrac{B(\alpha)}{1-\alpha}\epsilon\mathbb{E}_{\alpha}(\tau^\alpha),\; \tau\in J.
\end{equation}
By Lemma \ref{ABCLm2.2}, we find 
\begin{equation}\label{ABCe3.5}
^{ABC}{_0\mathcal{D}}_\tau^\alpha\left( \mathbb{E}_{\alpha}(\tau^\alpha)\right) =\;^{ABC}{_0\mathcal{D}}_\tau^\alpha\left( \mathbb{E}^1_{\alpha,1}(\tau^\alpha)\right) = \dfrac{B(\alpha)}{1-\alpha}\mathbb{E}^2_{\alpha,1}(\tau^\alpha),\;\tau\in J.
\end{equation}

\noindent Since $(2)_0=1$ and   $\dfrac{(2)_k}{k!}=k+1>1,\;k=1,2,\cdots$,  we have
\begin{equation}\label{ABCe3.6}
\mathbb{E}^2_{\alpha,1}(\tau^\alpha)= \sum_{k=0}^{\infty}\dfrac{z^k}{\Gamma(\alpha k+1)}\dfrac{(2)_k}{k!}\geq\sum_{k=0}^{\infty}\dfrac{z^k}{\Gamma(\alpha k+1)}=\mathbb{E}_{\alpha}(\tau^\alpha),\; \tau\in J.
\end{equation}
Using the  inequality \eqref{ABCe3.6}, Eq.\eqref{ABCe3.5} reduces to
\begin{equation}\label{e3.7}
^{ABC}{_0\mathcal{D}}_\tau^\alpha\left( \mathbb{E}_{\alpha}(\tau^\alpha)\right) \geq \dfrac{B(\alpha)}{1-\alpha}\mathbb{E}_{\alpha}(\tau^\alpha),~\tau\in J.
\end{equation}
Utilizing the inequalities (ii), \eqref{ABCe3.4} and \eqref{e3.7}, for any $\tau\in J$, we have
\begin{align*}
^{ABC}{_0\mathcal{D}}_\tau^\alpha \left( w_\epsilon(\tau)\right) &=\;^{ABC}{_0\mathcal{D}}_\tau^\alpha \left[ w(\tau)+\epsilon\mathbb{E}_{\alpha}(\tau^\alpha)\right]\\
&=\;^{ABC}{_0\mathcal{D}}_t^\alpha  w(\tau)+\epsilon\;^{ABC}{_0\mathcal{D}}_\tau^\alpha\mathbb{E}_{\alpha}(\tau^\alpha)\\
&\geq f(\tau,w(\tau)) +\epsilon\dfrac{B(\alpha)}{1-\alpha}\mathbb{E}_{\alpha}(\tau^\alpha)\\
&> f(\tau,w_\epsilon(\tau))-\dfrac{B(\alpha)}{1-\alpha}\epsilon\mathbb{E}_{\alpha}(\tau^\alpha)+\epsilon\dfrac{B(\alpha)}{1-\alpha}\mathbb{E}_{\alpha}(\tau^\alpha)\\
& =f(\tau,w_\epsilon(\tau))
\end{align*}
Therefore, 
$$
^{ABC}{_0\mathcal{D}}_\tau^\alpha w_\epsilon(\tau)>  f\left( \tau, w_\epsilon(\tau)\right),\; \tau\in J.
$$
Since $v(0)<w_\epsilon(0)$, by application of Theorem  \ref{ABCThm3.1} with $w(\tau)=w_\epsilon(\tau)$, for each $\epsilon>0$ we have
$$
v(\tau)<w_\epsilon(\tau),\; \tau\in J.
$$
Taking limit as $\epsilon\to 0$, in the above inequality and utilizing Eq. \eqref{ABCe3.3}, we obtain
$$
v(\tau)\leq w(\tau),\; \tau\in J.
$$
\end{proof}
% % % %
\begin{cor}
If $m$ is any differentiable function defined on $J$ such that $^{ABC}{_0\mathcal{D}}_\tau^\alpha m\in C(J)$, $\tau\in J$  and
\begin{align*}
^{ABC}{_0\mathcal{D}}_\tau^\alpha m(\tau)&\leq \dfrac{B(\alpha)}{1-\alpha}m(\tau),\;\tau\in J,\\ m(0)&=m_0,
\end{align*} then $m(\tau)\leq m_0\mathbb{E}_{\alpha}(\tau^\alpha),\; \tau\in J.$
\end{cor}
\begin{proof}
\noindent Define
$$
\lambda(\tau)=m_0\mathbb{E}_{\alpha}(\tau^\alpha),\; \tau\in J.
$$

\noindent Then $\lambda(0)=m_0$. Further using the inequality \eqref{e3.7}, we have 
$$
^{ABC}{_0\mathcal{D}}_\tau^\alpha \lambda(\tau)\geq\dfrac{B(\alpha)}{1-\alpha}\lambda(\tau),\; \tau\in J. 
$$
Applying Theorem \ref{ABCThm3.6}, with $v=m$ and $w=\lambda$, we get
$$
m(\tau)\leq \lambda(\tau)= m_0\mathbb{E}_{\alpha}(\tau^\alpha),\tau\in J.
$$
\end{proof} 

% % %
The following theorem is the comparison result for periodic boundary value problems involving ABC-fractional derivative. The proof of the same one can finish watching the comparable kind of steps of Theorem 2.6 \cite{Vats}.  
\begin{theorem}
Let  $v,w$ be any differentiable functions on $J$ such that $^{ABC}{_0\mathcal{D}}_\tau^\alpha\; v,^{ABC}{_0\mathcal{D}}_\tau^\alpha w\in C(J),\;\tau\in J$, and $f\in C(J\times\mathbb{R},\mathbb{R})$ satisfying
\item[\normalfont(i)] $^{ABC}{_0\mathcal{D}}_\tau^\alpha v(\tau)\leq  f\left( \tau, v(\tau)\right),\; \tau\in J,\; v(0)\leq v(T)$
\item[\normalfont(ii)] $^{ABC}{_0\mathcal{D}}_\tau^\alpha w(\tau)\geq  f\left( \tau, w(\tau)\right),\; \tau\in J,\; w(0)\geq w(T)$

\noindent  If the function $f(\tau,\omega) $ is non increasing in $\omega$ for each $\tau$ then
$$
v(\tau)\leq w(\tau),\; \tau\in J.
$$
\end{theorem}

In the next Theorem, we will give an alternating proof of Lemma 2.3 \cite{Moh} without utilizing  Cauchy-Schwartz inequality.

\begin{theorem}\label{ABCThm3.10}
If $m$ is any differentiable function defined on $J$ such that $^{ABC}{_0\mathcal{D}}_\tau^\alpha m\in C(J)$ then $^{ABC}{_0\mathcal{D}}_\tau^\alpha m(0)=0$. 
\end{theorem}
\begin{proof}
Let $m$ is any differentiable function defined on $J$ such that $^{ABC}{_0\mathcal{D}}_\tau^\alpha m\in C(J)$,\; $\tau\in J$.  Since $m'$ exists and is continuous on $J$, we have $m,m'\in C(J)$, this implies $m,m'\in L^1(J)$,  therefore
 $\underset{\tau\in[0,T]}{\sup}|m'(\tau)|\leq M$.
 Using the definition of ABC-fractional derivative operator,
 \begin{align*}
 |^{ABC}{_0\mathcal{D}}_\tau^\alpha m(\tau)|&\leq  \dfrac{B(\alpha)}{1-\alpha}\int_{0}^\tau\left| \mathbb{E}_\alpha\left(-\dfrac{\alpha}{1-\alpha} (\tau-\sigma)^\alpha\right)\right||m'(\sigma)|d\sigma\\
 &\leq\dfrac{MB(\alpha)}{1-\alpha}\int_{0}^\tau \mathbb{E}_\alpha\left(\dfrac{\alpha}{1-\alpha} (\tau-\sigma)^\alpha\right)d\sigma = \dfrac{MB(\alpha)}{1-\alpha}\mathbb{E}_{\alpha,2}\left(\dfrac{\alpha}{1-\alpha} \tau^\alpha\right)\tau
 \end{align*}
This gives,  $^{ABC}{_0\mathcal{D}}_\tau^\alpha m(0)=0$.
\end{proof}

\begin{cor}\label{cor1}
If $\omega,~^{ABC}{_0\mathcal{D}}_\tau^\alpha \omega,~u\in C(J)$ and $^{ABC}{_0\mathcal{D}}_\tau^\alpha \omega(\tau)=u(\tau)$, then $u(0)=0$.
\end{cor}
\begin{proof}
Proof follows from the Theorem \ref{ABCThm3.10}.
\end{proof}
\begin{rem}
From the {\normalfont Corollary \ref{cor1}} it follows that {\normalfont the ABC-FDEs }\eqref{ABCe1.1}--\eqref{ABCe1.2} and its equivalent fractional integral equation is consistent only if $f(0,\omega(0))=0$, where $\omega$ is differentiable function  with $^{ABC}{_0\mathcal{D}}_\tau^\alpha \omega\in C(J)$.
\end{rem}

% % % % % % % % % % % % % % % % % % % % % % % % % % % % % % % % % % % % % % % % % % % % % % % % % % % % % % % % % % % % %
\section{ Existence of Local and Extremal Solution}
In this section, we determine the results about the existence of local and extremal solutions of the ABC-FDEs \eqref{ABCe1.1}--\eqref{ABCe1.2} through the equivalent fractional integral equation given in the following Lemma.

\begin{Lemma}{\normalfont\cite{Ravi1,Syam}}\label{ABCLm4.1}
The equivalent fractional integral equation to the { \normalfont the  ABC-FDEs } \eqref{ABCe1.1}--\eqref{ABCe1.2}  is given by
\begin{equation*}
\omega(\tau)=\omega_0+\dfrac{1-\alpha}{B(\alpha)}  f(\tau,\omega(\tau)) +\dfrac{\alpha}{B(\alpha)\Gamma(\alpha)}\int_{0}^\tau(\tau-\sigma)^{\alpha-1}f(\sigma,\omega(\sigma))d\sigma.
\end{equation*}
\end{Lemma}

Let  $\epsilon>0$ be arbitrary. Consider the ABC-FDEs of the form,
\begin{align}
^{ABC}{_0\mathcal{D}}_\tau^\alpha \omega_\epsilon(\tau)&= f\left( \tau, \omega_\epsilon(\tau)\right), \tau\in J,\label{ABCe4.1}\\
\;\omega_\epsilon(\tau)|_{\tau=0}&=\omega_\epsilon(0),\label{ABCe4.2}
\end{align}
where  $\omega_\epsilon,\;^{ABC}{_0\mathcal{D}}_\tau^\alpha \omega_\epsilon\in C(J)$ and $f\in C(J\times \mathbb{R},\mathbb{R})$.  Then by Lemma \ref{ABCLm4.1}, the equivalent fractional integral equation of the  ABC-FDEs \eqref{ABCe4.1}-\eqref{ABCe4.2}  is given by 
\begin{equation*}
\omega_\epsilon(\tau)=\omega_\epsilon(0)+\dfrac{1-\alpha}{B(\alpha)}  f(\tau,\omega_\epsilon(\tau)) +\dfrac{\alpha}{B(\alpha)\Gamma(\alpha)}\int_{0}^\tau(\tau-\sigma)^{\alpha-1}f(\sigma,\omega_\epsilon(\sigma))d\sigma,\;\tau\in J.
\end{equation*}

\begin{theorem}\label{ABCThm4.2}
If the function $f\in C(R_0,\mathbb{R}), ~ R_0=\left\lbrace (\tau,\omega) : \tau\in J,|\omega-\omega_0|\leq b\right\rbrace$ is such that 
$$|f(\tau,\omega)|\leq M,\;\text{for all}\; (\tau,\omega)\in R_0$$ and satisfies the  Lipschitz type condition, 
$$
|f(\tau_1,\omega)-f(\tau_2,\eta)|\leq L_1|\tau_1-\tau_2|+L_2|\omega-\eta|,\;\tau_1,\tau_2\in J,\;\omega,\eta\in\mathbb{R},
$$
where $L_1>0$ and $ 0<L_2<\dfrac{B(\alpha)}{(1-\alpha)},$
then the family $\left\lbrace \omega_\epsilon\right\rbrace $ of solution of{ \normalfont the ABC-FDEs } \eqref{ABCe4.1}--\eqref{ABCe4.2} is equicontinious on $J$.
\end{theorem}
\begin{proof}
Let  $\epsilon>0$ be arbitrary. Let  $\omega_\epsilon(\tau)$ be the solution of the ABC-FDEs  \eqref{ABCe4.1}--\eqref{ABCe4.2}. Let $\tau_1,\tau_2\in J$ with $0< \tau_1\leq \tau_2< T$. Then by  hypotheses, we have  
\begin{align*}
&|\omega_\epsilon(\tau_1)-\omega_\epsilon(\tau_2)|\\
&=\left|\left( \omega_\epsilon(0)+\dfrac{1-\alpha}{B(\alpha)}  f(\tau_1,\omega_\epsilon(\tau_1)) +\dfrac{\alpha}{B(\alpha)\Gamma(\alpha)}\int_{0}^{\tau_1}(\tau_1-\sigma)^{\alpha-1}f(\sigma,\omega_\epsilon(\sigma))d\sigma\right) \right.\\&\qquad\left.-\left( \omega_\epsilon(0)+\dfrac{1-\alpha}{B(\alpha)}  f(\tau_2,\omega_\epsilon(\tau_2)) +\dfrac{\alpha}{B(\alpha)\Gamma(\alpha)}\int_{0}^{\tau_2}(\tau_2-\sigma)^{\alpha-1}f(\sigma,\omega_\epsilon(\sigma))d\sigma\right) \right| \\
&\leq\dfrac{1-\alpha}{B(\alpha)} |f(\tau_1,\omega_\epsilon(\tau_1))-f(\tau_2,\omega_\epsilon(\tau_2))|\\
&\quad+\dfrac{\alpha}{B(\alpha)\Gamma(\alpha)}\left( \int_{0}^{\tau_1}((\tau_1-\sigma)^{\alpha-1}-(\tau_2-\sigma)^{\alpha-1})|f(\sigma,\omega_\epsilon(\sigma))|d\sigma+\int_{\tau_1}^{\tau_2}(\tau_2-\sigma)^{\alpha-1}|f(\sigma,\omega_\epsilon(\sigma))|d\sigma\right) \\
&\leq \dfrac{1-\alpha}{B(\alpha)}\left(L_1|\tau_1-\tau_2|+L_2 |\omega_\epsilon(\tau_1)-\omega_\epsilon(\tau_2)|\right) +\dfrac{M\alpha}{B(\alpha)\Gamma(\alpha)}\int_{0}^{\tau_1}((\tau_1-\sigma)^{\alpha-1}-(\tau_2-\sigma)^{\alpha-1})d\sigma\\
&\quad+\dfrac{M\alpha}{B(\alpha)\Gamma(\alpha)}\int_{\tau_1}^{\tau_2}(\tau_2-\sigma)^{\alpha-1}d\sigma
\end{align*}
This gives, 
\begin{align*}
|\omega_\epsilon(\tau_1)-\omega_\epsilon(\tau_2)|\leq\dfrac{1}{\left( 1-\dfrac{1-\alpha}{B(\alpha)}L_2\right) }\left[   \dfrac{(1-\alpha)L_1}{B(\alpha)}(\tau_2-\tau_1)+\dfrac{M}{B(\alpha)\Gamma(\alpha)}\left( 2(\tau_2-\tau_1)^\alpha+\tau_1^\alpha-\tau_2^\alpha\right) \right].  
\end{align*}
Note that for any $0<\alpha<1,\;(\tau_2-\tau_1)^\alpha\leq (\tau_2-\tau_1)$ and $\tau_1^\alpha- \tau_2^\alpha\leq 0$. Therefore, we have
$$
|\omega_\epsilon(\tau_1)-\omega_\epsilon(\tau_2)|\leq \dfrac{\Gamma(\alpha)(1-\alpha)L_1+2M}{\Gamma(\alpha)\left\{ B(\alpha)-(1-\alpha)L_2\right\} }(\tau_2-\tau_1).
$$
One can check that for any $\tilde{\epsilon}>0$ there exists $\tilde{\delta}=\tilde{\epsilon}\dfrac{\Gamma(\alpha)\left\{ B(\alpha)-(1-\alpha)L_2\right\} }{\Gamma(\alpha)(1-\alpha)L_1+2M}$ such that if $|\tau_1-\tau_2|<\tilde{\delta}$, then 
$$
|\omega_\epsilon(\tau_1)-\omega_\epsilon(\tau_2)|<\tilde{\epsilon}.
$$
This proves that the family of solution $\left\lbrace \omega_\epsilon \right\rbrace$ of the ABC-FDEs \eqref{ABCe4.1}--\eqref{ABCe4.2} is equicontinious on $J$.
\end{proof}

\begin{theorem}\label{ABCThm4.3}
Assume that the conditions of Theorem \ref{ABCThm4.2} hold. If $M(1-\alpha)<bB(\alpha)$, then the ABC-FDEs \eqref{ABCe1.1}--\eqref{ABCe1.2} has at least one solution on $J'=[0,\beta]$, where $\beta=\min\left\lbrace  T,\left[ \dfrac{\Gamma(\alpha)(bB(\alpha)-M(1-\alpha))}{M} \right]^{\frac{1}{\alpha}} \right\rbrace $. 
\end{theorem}
\begin{proof}
Fix $\delta>0$. Let $\omega_0\in C[-\delta,0]$ be any  real valued function  satisfying the conditions
$$
\omega_0(0)=\omega_0,\;|\omega_0(\tau)-\omega_0|\leq b.
$$
For any $\epsilon$,\;  $0<\epsilon<\delta$, define $\beta_1=\min\left\lbrace \beta,\epsilon\right\rbrace $.  Consider the ABC-fractional delay  differential equations
\begin{align}
^{ABC}{_0\mathcal{D}}_\tau^\alpha \omega_\epsilon(\tau)&= f\left( \tau, \omega_\epsilon(\tau-\epsilon)\right), \tau\in [0,\beta_1], \label{ABCe4.3} \\
\omega_\epsilon(\tau)&=\omega_0(\tau),\tau\in[-\delta,0]\label{ABCe4.4}.
\end{align}

\noindent Then by Lemma \ref{ABCLm4.1}, equivalent fractional integral equ. of the ABC-FDEs \eqref{ABCe4.3}--\eqref{ABCe4.4} is 
\begin{equation}\label{ABCe4.5}
\omega_\epsilon(\tau)=\begin{cases}
\omega_0(\tau),\; \tau\in [-\delta,0],\\
\omega_0+\dfrac{1-\alpha}{B(\alpha)}  f(\tau,\omega_\epsilon(\tau-\epsilon)) \\
\;\;+\dfrac{\alpha}{B(\alpha)\Gamma(\alpha)}\large\int_{0}^\tau(\tau-\sigma)^{\alpha-1}f(\sigma,\omega_\epsilon(\sigma-\epsilon))d\sigma,\;\tau\in [0,\beta_1].
\end{cases}
\end{equation}
In the view of Corollary \ref{cor1}, the ABC-FDEs \eqref{ABCe4.3}--\eqref{ABCe4.4} is consistent only if 
\begin{equation}\label{ABCe4.6}
f(0,\omega_\epsilon(-\epsilon))=f(0,\omega_0(-\epsilon))=0
\end{equation}
One can observe that $\omega_\epsilon$ is continuous on $[-\delta,\beta_1]$ expect possibly at $\tau=0$.  Using continuity of $f$ and Eq.\eqref{ABCe4.6}, we have
\begin{align*}
\lim\limits_{\tau\to 0^+}\omega_\epsilon(\tau)&=\lim\limits_{\tau\to 0^+}\left( \omega_0+\dfrac{1-\alpha}{B(\alpha)}  f(\tau,\omega_\epsilon(\tau-\epsilon))
\;+\dfrac{\alpha}{B(\alpha)\Gamma(\alpha)}\large\int_{0}^\tau(\tau-\sigma)^{\alpha-1}f(\sigma,\omega_\epsilon(\sigma-\epsilon))d\sigma\right)\\
&=\omega_0+f(0,\omega_0(-\epsilon))=\omega_0
\end{align*} 
Hence the function  $\omega_\epsilon(\tau):[-\delta,\beta_1]\to\mathbb{R}$ is continuous. Note that,
 \begin{equation}\label{ABCe4.7}
|\omega_\epsilon(\tau)-\omega_0|=|\omega_0(\tau)-\omega_0|\leq b,\;\tau\in [-\delta,0].
\end{equation}
Also for any $\tau\in[0,\beta_1]$,
\begin{align}\label{ABCe4.8}
|\omega_\epsilon(\tau)-\omega_0|&\leq \dfrac{1-\alpha}{B(\alpha)} | f(\tau,\omega_\epsilon(\tau-\epsilon)) |+\dfrac{\alpha}{B(\alpha)\Gamma(\alpha)}\int_{0}^\tau(\tau-\sigma)^{\alpha-1}|f(\sigma,\omega_\epsilon(\sigma-\epsilon))|d\sigma\nonumber\\
&\leq\dfrac{(1-\alpha)M}{B(\alpha)} +\dfrac{M\alpha}{B(\alpha)\Gamma(\alpha)}\int_{0}^\tau(t-\sigma)^{\alpha-1}d\sigma\nonumber\\
&= \dfrac{M}{B(\alpha)}\left(1-\alpha+\dfrac{\tau^\alpha}{\Gamma(\alpha)} \right) \leq\dfrac{M}{B(\alpha)}\left(1-\alpha+\dfrac{\beta_1^\alpha}{\Gamma(\alpha)} \right).
\end{align}  
Since $\beta_1\leq\beta\leq \left[ \dfrac{\Gamma(\alpha)(bB(\alpha)-M(1-\alpha))}{M} \right]^{\frac{1}{\alpha}}$,  we have 
\begin{equation}\label{ABCe4.9}
\dfrac{\beta_1^\alpha}{\Gamma(\alpha)} \leq \dfrac{M}{B(\alpha)}\left(\dfrac{bB(\alpha)}{M}-1+\alpha\right)
\end{equation}
From \eqref{ABCe4.8} and \eqref{ABCe4.9}, we have
\begin{equation}\label{ABCe4.10}
|\omega_\epsilon(\tau)-\omega_0|\leq b, \tau\in[0,\beta_1]
\end{equation}
From \eqref{ABCe4.7} and \eqref{ABCe4.10}, we have
\begin{equation}
|\omega_\epsilon(\tau)-\omega_0|\leq b, \tau\in[-\delta,\beta_1].
\end{equation}
If $\beta_1<\beta$, we consider fractional integral Eq.\eqref{ABCe4.5} on the interval $[-\delta, \beta_2]$, where $\beta_2=\min\{\beta,2\epsilon\}$, such that
$$
|\omega_\epsilon(\tau)-\omega_0|\leq b,\;\tau\in[-\delta,\beta_2].
$$
Continuing in this way, $\omega_\epsilon(\tau) $ can be extended to $[-\delta,\beta]$, such that 
$$
|\omega_\epsilon(\tau)-\omega_0|\leq b,\;\tau\in[-\delta,\beta].
$$
This gives
$$
|\omega_\epsilon(\tau)|\leq |\omega_\epsilon(\tau)-\omega_0|+|\omega_0|\leq b+|\omega_0|,\;\tau\in[-\delta,\beta].
$$
Therefore
$$
\|\omega_\epsilon\|\leq b+|\omega_0|,\;\tau\in[-\delta,\beta], 
$$ 
and hence $\left\lbrace \omega_\epsilon\right\rbrace $ is uniformly bounded family of function defined on $[-\delta,\beta]$. Since, the hypothesis of Lemma \ref{ABCLm4.1} are satisfied, the  family $\left\lbrace \omega_\epsilon \right\rbrace $  is  equicontinious. Applying  Ascoli-Arzela's theorem there exists a decreasing sequence $\left\lbrace \epsilon_n\right\rbrace $  with $\epsilon_n>0$ for all $n$ and $\epsilon_n\to 0$ such that  $\omega_{\epsilon_n}\to \omega$, $\omega\in C([0,\beta],\mathbb{R})$ uniformly on   $[0,\beta]$. Since $f\in C(R_0,\mathbb{R})$, we have  $f(\tau,\omega_{\epsilon_n}(\tau-\epsilon_n))\to f(\tau,\omega)$. By replacing  $\epsilon$ with $\epsilon_n$ in equation \eqref{ABCe4.5} and then taking limit as $n\to\infty$, we obtain
$$
\omega(\tau)=\omega_0+\dfrac{1-\alpha}{B(\alpha)}  f(\tau,\omega(\tau)) +\dfrac{\alpha}{B(\alpha)\Gamma(\alpha)}\int_{0}^\tau(\tau-\sigma)^{\alpha-1}f(s,\omega(s))d\sigma,\;\tau\in [0,\beta],
$$
which gives the required solution of the ABC-FDEs \eqref{ABCe1.1}--\eqref{ABCe1.2}.
\end{proof}

\begin{theorem}\label{ABCThm4.4}
Assume that the conditions of Theorem \ref{ABCThm4.2} hold. If $(2M+b)(1-\alpha)< bB(\alpha)$, then the ABC-FDEs \eqref{ABCe1.1}--\eqref{ABCe1.2} has extremal solution on $J''=[0,\beta_0]$, where $\beta_0=\min\left\lbrace  T,\left[ \dfrac{ (bB(\alpha)-(2M+b)(1-\alpha))\Gamma(\alpha)}{(2M+b)}\right]^\frac{1}{\alpha}\right\rbrace $.
\end{theorem}
 \begin{proof}
We give the proof only for the existence of maximal solution of the ABC-FDEs \eqref{ABCe1.1}--\eqref{ABCe1.2}, as the proof of existence of minimal solution one can complete on similar lines. 
 
\noindent For $0<\epsilon\leq\frac{b}{2},$ consider 
the ABC-FDEs
\begin{align}
^{ABC}{_0\mathcal{D}}_\tau^\alpha \omega(\tau)&= f\left( \tau, \omega(\tau)\right)+\epsilon:=f_\epsilon\left( \tau, \omega(\tau)\right), \tau\in J, \label{ABCe4.12} \\
\omega(0)&=\omega_0+\epsilon=\omega(0,\epsilon).\label{ABCe4.13}
\end{align}

\noindent Define  $R_\epsilon=\left\lbrace (\tau,\omega):\tau \in J,~|\omega-\omega(0,\epsilon) |\leq \dfrac{b}{2}\right\rbrace $. Clearly  $R_\epsilon\subset R_0.$ Consider the function  $f_\epsilon: R_\epsilon\to\mathbb{R}$ defined by
$$
f_\epsilon\left( \tau, \omega(\tau) \right)=f\left( \tau, \omega(\tau) \right)+\epsilon.
$$ 
Then, $f_\epsilon$ satisfies the Lipschitz type condition with same Lipschitz constants $L_1, L_2$ as defined in Theorem \ref{ABCThm4.2}. Further,
$$
|f_\epsilon\left( \tau, \omega(\tau) \right)|\leq M+\frac{b}{2}, \; (\tau,\omega)\in R_\epsilon.
$$

\noindent Since $f_\epsilon$  satisfies  all assumptions of Theorem \ref{ABCThm4.3}, by applying it, the ABC-FDEs \eqref{ABCe4.12}--\eqref{ABCe4.13} has at least one  solution $\omega(\tau,\epsilon)$ on $J''$.

\noindent Let $0<\epsilon_2<\epsilon_1\leq \epsilon$, then we have
\begin{align*}\label{eABCe4.7}
\omega(0,\epsilon_2)=\omega_0+\epsilon_2&<\omega_0+\epsilon_1=\omega(0,\epsilon_1);\\
^{ABC}{_0\mathcal{D}}_\tau^\alpha \omega(\tau,\epsilon_2)&= f\left( \tau, \omega(\tau,\epsilon_2)\right)+\epsilon_2,\;\tau\in J'';\\
^{ABC}{_0\mathcal{D}}_\tau^\alpha \omega(\tau,\epsilon_1)&> f\left( \tau, \omega(\tau,\epsilon_1)\right)+\epsilon_2,\;\tau\in J''.
\end{align*}  
Note that $\omega(\tau,\epsilon_1)$ and $\omega(\tau,\epsilon_2)$ respectively are lower and upper solutions of ABC-FDEs, with $\omega(0,\epsilon_2)<\omega(0,\epsilon_1)$. Therefore by applying Theorem \ref{ABCThm3.5} we have, 
$$
\omega(\tau,\epsilon_2)<\omega(\tau,\epsilon_1),\; \tau\in J''.
$$

\noindent Next we show that  the family  $\left\lbrace \omega(\tau,\epsilon)\right\rbrace $ of solutions of the ABC-FDEs \eqref{ABCe4.12}--\eqref{ABCe4.13} is uniformally bounded on $J''$. Proceeding as in the proof of Theorem \ref{ABCThm4.3}, for any $\tau\in J''$, we have
\begin{align*}
&|\omega(\tau,\epsilon)-\omega(0,\epsilon)|\leq \dfrac{2M+b}{2B(\alpha)}\left( 1-\alpha+\dfrac{\beta_0^\alpha}{\Gamma(\alpha)}\right), \tau\in J''.
\end{align*}
Since 
$\beta_0\leq \left[ \dfrac{ (bB(\alpha)-(2M+b)(1-\alpha))\Gamma(\alpha)}{(2M+b)}\right]^\frac{1}{\alpha},$ above inequality reduces to
$$
|\omega(\tau,\epsilon)-\omega(0,\epsilon)|\leq \dfrac{2M+b}{2B(\alpha)}\left( 1-\alpha+\dfrac{bB(\alpha)}{2M+b}-1+\alpha\right) \\
\leq \dfrac{b}{2}<b,\; \tau\in J''. 
$$
Since  $f_\epsilon$ satisfies assumptions of Theorem \ref{ABCThm4.2},  the family  $\left\lbrace \omega(\tau,\epsilon)\right\rbrace $ is equicontinious on $J$. Applying  Ascoli-Arzela's theorem there exists a decreasing sequence $\left\lbrace \epsilon_n\right\rbrace $  with $\epsilon_n>0$ for all $n$ and $\epsilon_n\to 0$ such that  $\omega(\tau,\epsilon_n)\to \eta(\tau)$ uniformly on $J'' $, where $\eta\in C([0,\beta_0],\mathbb{R})$ uniformly on   $[0,\beta]$. By uniform continuity of $f_\epsilon$, we have
$$
f_{\epsilon_n}(\tau,\omega(\tau,\epsilon_n))\to f(\tau,\eta(\tau)),\; \tau\in J''.
$$ 
By replacing  $\epsilon$ with $\epsilon_n$ in equation \eqref{ABCe4.12} and then taking limit as $n\to\infty$, we obtain
 \begin{align*}
 ^{ABC}{_0\mathcal{D}}_\tau^\alpha \eta(\tau)&= f\left( \tau, \eta(\tau)\right),\tau\in J'',\\
 \eta(0)&=\omega_0.
 \end{align*}
 This proves that $\eta(\tau)$ is a solution of ABC-FDEs\eqref{ABCe1.1}--\eqref{ABCe1.2}. 
 
 It remains to prove that $\eta(\tau)$ is the maximal solution of the  ABC-FDEs \eqref{ABCe1.1}--\eqref{ABCe1.2}.   Let $\omega(\tau)$ be any solution of \eqref{ABCe1.1}--\eqref{ABCe1.2} on $J''$. Then for any $\epsilon>0$,
\begin{align*}
\omega_0=\omega(0)&<\omega_0+\epsilon=\omega(0,\epsilon);\\ 
^{ABC}{_0\mathcal{D}}_\tau^\alpha \omega(\tau)&> f\left( \tau, \omega(\tau,\epsilon_2)\right)+\epsilon,\;\tau\in J'';\\
^{ABC}{_0\mathcal{D}}_\tau^\alpha \omega(\tau,\epsilon)&= f\left( \tau, \omega(\tau,\epsilon)\right)+\epsilon,\;\tau\in J''.
\end{align*}
Note that $\omega(\tau)$ and $\omega(\tau,\epsilon)$ respectively are lower and upper solutions of ABC-FDEs, with $\omega_0<\omega_0+\epsilon$. Therefore by applying Theorem \ref{ABCThm3.5} we have,
$$
 \omega(\tau)<\omega(\tau,\epsilon),\; \tau\in J''.
$$
Taking limit as $\epsilon\to 0$, we obtain
$$
\omega(\tau)\leq \eta(\tau),\; \tau\in J''.
$$
 This proves that $\eta(\tau)$ is the maximal solution of the ABC-FDEs \eqref{ABCe1.1}--\eqref{ABCe1.2} on $J"$.
\end{proof}

% % % % % % % % % % % % % % % % % % % % % % % % % % % % % % % % % % % % % % % % % % % % % % % % % % % %
\section{Existence of Global Solution}
\begin{theorem}\label{ABCThm5.1}
Assume that $m$ is any differentiable function defined on $J$ such that $ ^{ABC}{_0\mathcal{D}}_\tau^\alpha m\in C(J)$ and $g\in C(J\times\mathbb{R}_+,\mathbb{R}_+)$ such that 
\begin{equation}\label{e4.1}
^{ABC}{_0\mathcal{D}}_\tau^\alpha m(\tau)\leq g(\tau,m(\tau)) ,\; \tau\in J
\end{equation}
and let $\eta(\tau)$ is the maximal solution of the ABC-FDEs 
\begin{align*}
^{ABC}{_0\mathcal{D}}_\tau^\alpha  u(\tau)&= g\left( \tau, u(\tau)\right), \tau\in J,\;  \\
u(0)&=u_0.
\end{align*}
Then $m(0)\leq u(0)$, implies  $m(\tau)\leq \eta(\tau),\;\tau\in J.$
\end{theorem}
\begin{proof}
Let $\epsilon>0$ be arbitrary. Let $u(\tau,\epsilon)$ be a solution of the ABC-FDEs of the form
\begin{align*}
^{ABC}{_0\mathcal{D}}_\tau^\alpha u(\tau)&= g\left( \tau, u(\tau)\right)+\epsilon, \tau\in J,\\
u(0)&=u_0+\epsilon.
\end{align*}
Therefore,
\begin{align}\label{e5.2}
^{ABC}{_0\mathcal{D}}_\tau^\alpha u(\tau,\epsilon)&> g\left( \tau, u(\tau,\epsilon)\right), \tau\in J,\nonumber\\
u(0)&=u_0+\epsilon
\end{align}
Therefore $m$ is lower solution and $u(\tau,\epsilon)$ is upper solution of the ABC-FDE
$$
^{ABC}{_0\mathcal{D}}_\tau^\alpha \omega(\tau)= g\left( \tau, \omega(\tau)\right), \tau\in J.
$$ 
Further,
$$
m(0)\leq u(0)<u(0)+\epsilon=u(0,\epsilon).
$$
By applying Theorem \ref{ABCThm3.5}, we obtain
$$
m(\tau)\leq u(\tau,\epsilon),\;\tau\in J,\;\epsilon>0.
$$
Taking $\epsilon\to 0$, and following similar approach as in the proof of Theorem \ref{ABCThm4.4}, we obtain
$$
m(\tau)\leq \eta(\tau),\;\tau\in J.
$$
\end{proof}
% % % % %
\begin{theorem}\label{ABCThm5.2}
Assume that $f\in C([0,\infty)\times\mathbb{R},\mathbb{R})~\mbox{and}~\;g\in C([0,\infty)\times\mathbb{R}_+,\mathbb{R}_+)$ are  such that $|f(\tau,\omega)|\leq g(\tau,|\omega|)$. Further, suppose that there exists local  solution $\omega(\tau,\omega_0)$ of the ABC-FDEs
\begin{align}
^{ABC}{_0\mathcal{D}}_\tau^\alpha \omega(\tau)&= f\left( \tau, \omega(\tau)\right), \tau\in [0, \infty),\label{ABCe5.3} \\
\omega(0)&=\omega_0,\label{ABCe5.4}
\end{align}
and maximal solution $\eta(\tau)$ of 
\begin{align*}
^{ABC}{_0\mathcal{D}}_\tau^\alpha u(\tau)&= g\left( \tau, u(\tau)\right), \tau\in[0,  \infty), \\
u(0)&=u_0\geq 0.
\end{align*}
Then the largest interval of existence of solution $\omega(\tau,\omega_0)$ of the  \eqref{ABCe5.3}--\eqref{ABCe5.4} such that $|\omega_0|<u_0$ is $[0,\infty)$.
\end{theorem}
\begin{proof}
By assumption there exists a local solution $\omega(\tau,\omega_0)$ of the \eqref{ABCe5.3}--\eqref{ABCe5.4} on the interval $[0,\beta)$, where $\beta<\infty$ with $|\omega_0|<u_0$. Suppose on contrary $\beta$ can not be increased further.

\noindent Define $m(\tau)=|\omega(\tau,\omega_0)|,\;\tau\in[0, \beta)$.  As $|f|'\leq |f'|$, for any $ \tau\in[0, \beta)$, we find
\begin{align*}
^{ABC}{_0\mathcal{D}}_\tau^\alpha m(\tau)&=\dfrac{B(\alpha)}{1-\alpha}\int_{0}^\tau \mathbb{E}_\alpha\left(-\dfrac{\alpha}{1-\alpha} (\tau-\sigma)^\alpha\right)m'(\sigma)d\sigma\\
&= \dfrac{B(\alpha)}{1-\alpha}\int_{0}^\tau \mathbb{E}_\alpha\left(-\dfrac{\alpha}{1-\alpha} (\tau-\sigma)^\alpha\right)|\omega(\sigma,\omega_0|'d\sigma\\
&\leq \dfrac{B(\alpha)}{1-\alpha}\int_{0}^\tau \mathbb{E}_\alpha\left(-\dfrac{\alpha}{1-\alpha} (\tau-\sigma)^\alpha\right)|\omega'(\sigma,\omega_0|d\sigma\\
&=\left| \dfrac{B(\alpha)}{1-\alpha}\int_{0}^\tau \mathbb{E}_\alpha\left(-\dfrac{\alpha}{1-\alpha} (\tau-\sigma)^\alpha\right)\omega'(\sigma,\omega_0)d\sigma\right| \\
&=|^{ABC}{_0\mathcal{D}}_\tau^\alpha \omega(\tau,\omega_0)|
\end{align*}
Therefore,
$$^{ABC}{_0\mathcal{D}}_\tau^\alpha m(\tau)\leq|^{ABC}{_0\mathcal{D}}_\tau^\alpha \omega(\tau,\omega_0)|=|f(\tau,\omega(\tau,\omega_0))|\leq g(\tau,|\omega(\tau,\omega_0)|)=g(\tau,m(\tau)),\; \tau\in [0,\beta).
$$

\noindent This implies $m(\tau)$ is lower solution of
\begin{align}
^{ABC}{_0\mathcal{D}}_\tau^\alpha u(\tau)&= g\left( \tau, u(\tau)\right), \tau\in[0,  \beta). \label{ABCe5.5} 
\end{align}
Further by assumption $\eta(\tau)$ is the maximal solution of the problem  \eqref{ABCe5.5}. Since $m(0)= |\omega_0|\leq u_0$, by applying Theorem \ref{ABCThm3.5}, we obtain
$$
m(\tau)=|\omega(\tau,\omega_0)|\leq \eta(\tau),\; \tau\in [0,\beta).
$$ 
By assumption $\eta(\tau)$ exists on  $[0,\infty)$. Therefore by continuity of $g$ on $[0,\beta]\times\mathbb{R}_+$,
 there exists $M>0$ such that, 

$$
|g\left( \tau, \eta(\tau)\right)|\leq M, \;  \tau\in [0,\beta). 
$$
Let $0\leq \tau_1\leq \tau_2< \beta$.  Then  following similar steps as in the proof of  Theorem \ref{ABCThm3.5}, we have
$$
|\omega(\tau_1,\omega_0)-\omega(\tau_2,\omega_0)|\leq \dfrac{\Gamma(\alpha)(1-\alpha)L_1+2M}{\Gamma(\alpha)\left[ B(\alpha)-(1-\alpha)L_2\right] }(\tau_2-\tau_1).
$$ 
 From above inequality it follows that, 
 $$
 \lim\limits_{\tau_1,\; \tau_2\to \beta^-}\omega(\tau_1,\omega_0)=\lim\limits_{\tau_2,\; \tau_1\to \beta^-}\omega(\tau_2,\omega_0),
 $$ 
 for any $\tau_1,\tau_2$ with $0\leq \tau_1\leq \tau_2< \beta$. This implies that $\lim\limits_{\tau\to \beta^-}\omega(\tau,\omega_0)$ exists. 
Let
$$
\omega(\beta, \omega_0)=\lim\limits_{\tau\to \beta^-}\omega(\tau,\omega_0).
$$ 
Then by assumption the ABC-FDEs
\begin{align*}
^{ABC}{_0\mathcal{D}}_\tau^\alpha \omega(\tau)&= f\left( \tau, \omega(\tau)\right),\tau\geq \beta,   \\
\omega(\beta)&=\omega_\beta,
\end{align*}
has a local solution. This implies that   $\omega(\tau,\omega_0)$ can be continued beyond $\beta$ which is a  contradiction our assumption. Hence every solution  $\omega(\tau,\omega_0)$   of the ABC-FDEs \eqref{ABCe5.3}--\eqref{ABCe5.4} exists on $[0,\infty)$.
\end{proof}

\section*{Conclusion}
The comparison results, local, extremal, and global existence of a solution, derived without demanding the monotonicity and Holder continuity assumption on the nonlinear function involving ABC-FDEs. The estimations on ABC-fractional derivative and the comparison results obtained to ABC-FDEs one can use to research the existence, uniqueness  and qualitative properties of solutions for various classes of ABC-FDEs.

\end{document}